\documentclass[a4paper,11pt]{amsart}

\usepackage{amsmath,amssymb,amsbsy,amsthm,amsfonts}

\def\eqdefa{\buildrel\hbox{\footnotesize def}\over =}

\newcommand{\beq}{\begin{eqnarray}}
\newcommand{\eeq}{\end{eqnarray}}
\newcommand{\bq}{\begin{equation}}
\newcommand{\eq}{\end{equation}}
\newcommand{\beqn}{\begin{eqnarray*}}
\newcommand{\eeqn}{\end{eqnarray*}}

\DeclareMathSymbol{\subsetneqq}{\mathbin}{AMSb}{36}

\newcommand{\C}{\mathbb{C}}
\newcommand{\R}{\mathbb{R}}
\newcommand{\N}{\mathbb{N}}

\newcommand{\dint}{\displaystyle\int}

\newtheorem{th1}{{\bf Theorem}}[section]
\newtheorem{thm}[th1]{{\bf Theorem}}
\newtheorem{lem}[th1]{{\bf Lemma}}
\newtheorem{prop}[th1]{{\bf Proposition}}

\theoremstyle{remark}
\newtheorem{rem}[th1]{\bf Remark}

\theoremstyle{definition}
\newtheorem{defi}[th1]{\bf Definition}

\author{Mohamed Majdoub}
\address{Universit\'e de Tunis El Manar,
Facult\'e des Sciences de Tunis, D\'epartement de Math\'ematiquess, 2092, Tunis, Tunisie}
\email{\sl mohamed.majdoub@fst.rnu.tn}
\thanks{{\sf M. Majdoub} is grateful to the Laboratory of
PDE and Applications at the Faculty of Sciences of Tunis}

\author{Nader Masmoudi}
\address{Courant Institute of Mathematical Sciences, New York, NY}
\email{\sl masmoudi@courant.nyu.edu}
\thanks{{\sf N. Masmoudi} is partially supported by an NSF Grant DMS-1211806}

\title[On Uniqueness for supercritical ...]
{On Uniqueness for supercritical nonlinear wave and Schr\"odinger equations}
\date{\today}

\begin{document}


\subjclass[2000]{35L05, 35Q55, 74G30, 35D30, 35A09}
\keywords{Nonlinear wave equation, nonlinear Schr\"odinger equation, supercritical growth, weak-strong uniqueness.}

\maketitle


\section{Introduction}

In a recent paper \cite{Struwe-IMRN}, Struwe considered the Cauchy problem for a class of nonlinear wave and Schr\"odinger equations.
 Under some assumptions on the nonlinearities, it was shown that uniqueness of classical solutions can be obtained in the
 much larger class of distribution solutions satisfying the energy inequality. As pointed out in \cite{Struwe-IMRN}, the conditions
 on the nonlinearities are satisfied for any polynomial growth but they fail to hold for higher growth (for example ${\rm e}^{u^2}$).
 Our aim here is to improve Struwe's result by showing that uniqueness holds for more general nonlinearities including higher
 growth or oscillations.\\

We briefly recall known results on the uniqueness of solution for some evolutionary PDEs. The strong solution is often constructed
within the framework of $C([0, T]; X)$ and an auxiliary space related to a priori estimates (Strichartz spaces, smoothing spaces, ...) using fixed point argument.
This, of course, yields uniqueness in the space where the fixed point argument is performed. This kind of uniqueness is called
conditional uniqueness. The uniqueness without an auxiliary space is called the unconditional uniqueness.
We refer among others to \cite{MP, P} in the case of wave equation, to \cite{Tsut} in the case of Schr\"odinger equation, to \cite{FK} in the case of Navier-Stokes system, to \cite{BP} in the case of Benjamin-Ono equation, to \cite{MK1, MK2} in the case of Zakharov systems and Maxwell-Dirac equation, to \cite{Pecher} in the case of Klein-Gordon-Schr\"odinger system and to \cite{Selberg} in the case of Dirac-Klein-Gordon equations.\\

In many cases we can only construct global weak solutions by using compactness arguments, but their uniqueness is not known.
The weak-strong uniqueness investigation is an attempt to reconcile the weak and strong viewpoints of solutions.
More precisely, this investigation is to show that any weak solution agrees with the strong solution sharing the same
 initial data if it exists. See, for instance, \cite{Stru2, Struwe-IMRN} for wave and Schr\"odinger equations, \cite{Chemin, Germ3, Leray, May} for the Navier-Stokes system, \cite{Germ1} for nonhomogeneous Navier-Stokes system and \cite{Germ2} for the isentropic compressible Navier-Stokes system.\\

The rest of the paper is organized as follows. In the next section, we consider nonlinear wave equation with two
kinds of supercritical nonlinearities (defocusing or oscillating). In both cases, we prove weak-strong uniqueness.
Section 3 is devoted to nonlinear Schr\"odinger equation. We also obtain weak-strong uniqueness.
Finally, in the Appendix we recall the global existence of weak solutions for nonlinear wave equation
assuming that the nonlinearity oscillates.


\section{Nonlinear Wave Equations}
Consider the Cauchy problem
\begin{equation}
\label{wave1}
\partial_t^2u-\Delta\,u+f(u)=0\quad\mbox{on}\quad \R_t\times\R_x^d,\;\;d\geq 2,
\end{equation}
with data
\begin{equation}
\label{CD}
\left(u(0,\cdot),u_t(0,\cdot)\right)=\left(u_0,u_1\right),
\end{equation}
where $f=F'$ for some smooth function $F:\R\to\R$ satisfying $F(0)=F'(0)=F^{''}(0)=0$. Moreover, we request either ({\sf defocusing case})
\bq
\label{H1}
u\,f(u)\geq 0\;\mbox{for all}\;u\in\R,
\eq

or ({\sf oscillating case})

\bq
\label{H2}
|f(u)|\,\leq\,C|u|^q\;\;\mbox{for some}\;\; q<2^*:=\frac{2d}{d-2}\;\;\;(2^*=+\infty\;\;\mbox{if}\;\;d=2),
\eq

{\sf and}

\bq
\label{H21}
F(u)\geq -C|u|^2\,.
\eq

The assumption \eqref{H1} allows a wide range of supercritical defocusing growths. A typical example
satisfying \eqref{H1} is $F(u)={\rm e}^{u^{2m}}-1$ where $m\geq 1$ is an integer. The hypotheses \eqref{H2}-\eqref{H21} allow slightly
oscillating supercritical powers. A good prototype is $F(u)= \sin\left(u|u|^q\right)$ where $1\leq q<2^*$.
Note that classical hypotheses (see \cite{GWE, Strauss}) do not allow these types of functions.\\

Multiplying \eqref{wave1} by $u_t$ and integrating over $\R^d$, one easily finds the energy identity
\begin{equation}
\label{energy}
E(u,t)=\int_{\R^d}\,\Big\{\frac{1}{2}\,|{\mathbf D}u(t,x)|^2+F(u(t,x))\Big\}\,dx=E(u,0),
\end{equation}
where ${\mathbf D}u=(u_t,\nabla u)$ is the space-time gradient.

We begin by defining what we call {\it classical}/{\it weak} solution of the Cauchy problem \eqref{wave1}-\eqref{CD}.

\begin{defi}\quad\\
\vspace{-0.5cm}
\label{solution}
\begin{itemize}
\item A function $u$ is a {\it classical} solution of \eqref{wave1}-\eqref{CD} on a time interval $I$ containing $0$
 if $u\in C^2(I\times\R^d)$ and solves \eqref{wave1}-\eqref{CD} in the classical sense.
\item A function $u$ is a {\it weak} solution of \eqref{wave1}-\eqref{CD} if $u\in L^2(\R\times\R^2)$,
${\mathbf D}u\in L^\infty(\R, L^2)$, $f(u)\in L^1(\R\times\R^d)$, $F(u)\in L^\infty(\R; L^1(\R^d))$, $u$
 solves \eqref{wave1}-\eqref{CD} in distribution sense, namely
\begin{eqnarray}
\label{weaksense}
\nonumber\int_0^{+\infty}\int_{\R^d}\,u\Box\varphi\,dx\,dt&+&\int_0^{+\infty}\int_{\R\R^d}\,f(u)\varphi\,dx\,dt=
-\int_{\R^d}\,u_0(x)\partial_t\varphi(0,x)\,dx\\&+&
\int_{\R^d}\,u_1(x)\varphi(0,x)\,dx,\;\;\;\forall\;\;\;\varphi\in{\mathcal D}(\R\times\R^d),
\end{eqnarray}
and the energy inequality \eqref{energy1} holds. Here $\Box=\partial_t^2-\Delta_x$ is the d'Alembertian operator.
\end{itemize}
\end{defi}
As explained in \cite{GWE}, for any weak solution $u$ in the sense of Definition \ref{solution}
the map $t\longmapsto {\mathbf D}u(t)\in L^2(\R^d)$ is weakly continuous.\\

The Cauchy problem \eqref{wave1}-\eqref{CD} has a long history. First, when the space dimension $d\geq 3$,
 the defocusing semilinear wave equation with power $p$ reads
\begin{equation}
\label{N} \partial_t^2u-\Delta u+|u|^{p-1}u=0,
\end{equation}
where $p>1$. This problem has been widely
investigated and there is a large literature dealing with the
well-posedness theory of (\ref{N}) in the scale of the Sobolev
spaces $H^s$. Second, for the global solvability in the energy space
$\dot{H}^1\times L^2$, there are mainly three cases.
In the subcritical case $p<p^*:=\frac{d+2}{d-2}$, Ginibre and Velo \cite{GV1} finally settled
 global well-posedness in the energy space, by using the Strichartz estimate, nonlinear estimates in Besov space,
 and energy conservation.

The critical case $p=p^*$ is more delicate, due to possibility of energy concentration. Struwe \cite{Str88} proved
global existence of radially symmetric regular solutions.
 Then Grillakis \cite{Gr1, Gr2} extended this result to non-radial data. In the energy space, Ginibre, Soffer and Velo \cite{GSV}
 proved global well-posedness in the radial case, where the Morawetz estimate effectively precludes concentration.
  The case of general data was solved by Shatah-Struwe \cite{SS2}, and Kapitanski \cite{K}. See also Ibrahim-Majdoub \cite{IM}
  for variable metrics. Note that uniqueness in the energy space is not yet fully solved.
   We refer to \cite{P} for $d\ge 4$, \cite{Stru2,MP} for partial results in $d=3$ and to \cite{Struwe-IMRN} for the case of
   classical solutions.

The supercritical case $p>p^*$ is even harder, and the global well-posedness problem for general data remains open,
except for the existence of global weak solutions \cite{Strauss}, local wellposedness in higher Sobolev spaces ($H^s$
 with $s\ge d/2-2/p>1$) as well as global wellposedness with scattering for small data (see e.g.~\cite{LiSo,Wang}), and
some negative results concerning non-uniform continuity of the solution map \cite{CCT, Le}. See also \cite{Leb} for a
result concerning a loss of regularity and \cite{tao} for a result about global regularity for a logarithmically energy-supercritical
 wave equation in the radial case. See also \cite{BT08} about random data Cauchy theory for supercritical wave equations.

In dimension two, $H^1$-critical nonlinearities are of
exponential type {\footnote{In fact, the critical nonlinearity is of
exponential type in any dimension $d$ with respect to $H^{d/2}$
norm.}}, since every power is $H^1$-subcritical. On the one hand, in a recent work \cite{IMM1}, the case
 $F(u)=\frac{1}{8\pi}\left({\rm e}^{4\pi u^2}-1\right)$ was investigated and an energy threshold was proposed.
 Local strong well-posedness was shown under the size restriction $\|\nabla u_0\|_{L^2}<1$ and the global
 well-posedness was obtained when the energy is below or equal to the energy threshold.
 Recently, Struwe \cite{Struwe09, Struwe2D} has constructed global smooth solutions with smooth data of arbitrary size.
 On the other hand, the ill posedness results of \cite{CCT, Leb}
show the non uniform continuity of the solution map (or sometimes its non continuity at the zero data).
 In the 2D exponential case the non uniform continuity of the solution map was shown in \cite{IMM2, APDE}.\\

Assuming \eqref{H1} or \eqref{H2}-\eqref{H21}, we can construct a global weak solution with finite energy data
such that
\begin{equation}
\label{energy1}
E(u(t))\leq E(u(0)),\;\;\; \forall\;\;t\in\R\,.
\end{equation}
This was carried out by Strauss \cite{Strauss} in the case \eqref{H1} (see also \cite{Lions}, \cite{Segal}, \cite{GWE}).
 The case \eqref{H2}-\eqref{H21} can be done in a similar way. For the convenience of the reader, we give the proof in the Appendix.
\begin{rem}
To our knowledge, this type of oscillating supercritical nonlinearities (satisfying \eqref{H2}-\eqref{H21}) are not
studied in the literature. It is an interesting open question to prove the {\it global} well-posedness for such nonlinearities.
\end{rem}
Now we are ready to state our main result concerning the Cauchy problem \eqref{wave1}-\eqref{CD}.
\begin{thm}
\label{main-wave}
Suppose $u\in {\mathcal C}^2([-T,T]\times\R^d)$ is a classical solution to \eqref{wave1}-\eqref{CD} with Cauchy data in
${\mathcal D}(\R^d)^2$.
Also let $v$ be a weak solution to \eqref{wave1}-\eqref{CD}. Then $u=v$ on $[-T,T]$.
\end{thm}
\begin{rem}
For nonlinearities with arbitrary (supercritical) growth, we don't know wether the Cauchy problem
\eqref{wave1}-\eqref{CD} admits or not global smooth solutions.
\end{rem}
The underlying idea of our proof in the case when $f$ satisfies assumption \eqref{H1} consists of the following identity
satisfied by any weak solution.
\begin{prop}
\label{weak}
Assume that the nonlinearity $f$ satisfies \eqref{H1}. Let $v$ be weak
solution to \eqref{wave1}-\eqref{CD}. Then, for any $T>0$,
\begin{equation}
\label{weak1}
\int_0^T\,\int_{\R^d}\,\Big(|\nabla v|^2-|\partial_t v|^2+vf(v)\Big)\,dx\,dt=\int_{\R^d}\,
\Big(\partial_t v(T)v(T)-\partial_t v(0)v(0)\Big)\,dx\,.
\end{equation}
\end{prop}
\begin{proof}[Proof of Proposition \ref{weak}]
Let $(\rho_n)$ be a mollifier sequence in $\R^d$ and, for $k\in\N$, define the $C^1$ odd function $\beta_k :\R\to \R$ by
\begin{eqnarray*}
 \beta_k(s)&=&\; \left\{
\begin{array}{cllll}s \quad&\mbox{if}&\quad
0\leq s\leq k,\\ s+\frac{(s-k)^2}{2k}\quad
&\mbox{if}&\quad k\leq s\leq 2k ,\\
\frac{3k}{2}\quad&\mbox{if}&\quad s\geq 2k.
\end{array}
\right.
\end{eqnarray*}
Using \eqref{weaksense} with $v_{k,n}\eqdefa\beta_k(v)\ast\rho_n\eqdefa v_k\ast\rho_n$ as test function
\bq
\label{k,n}
\int_0^T\int_{\R^d}\,\left(\nabla v\cdot\nabla v_{k,n}-\partial_t v\,\partial_t v_{k,n}+f(v) \,v_{k,n}\right)\,dx dt=
\int_{\R^d}\left(\partial_t v\, v_{k,n}(T)-\partial_t v\, v_{k,n}(0)\right)\,dx dt\,.
\eq
Since $\nabla v_{k,n}=\nabla v_k\ast\rho_n\to \nabla v_k$ in $L^2(\R^d)$ when $n$ go to infinity, we deduce that
\bq
\label{k,n1}
\nabla v\cdot\nabla v_{k,n}\to \nabla v\cdot\nabla v_k\quad\mbox{in}\quad L^1(\R^d)\quad (n\to\infty)\,.
\eq
Taking advantage of
\bq
\label{k,n2}
\int_{\R^d}\,|\nabla v\cdot\nabla v_{k,n}|\,dx\leq \|\nabla v\|_{L^2}^2\in L^1(0,T),
\eq
the Lebesgue theorem implies
\bq
\label{k,n3}
\int_0^T\int_{\R^d}\,\nabla v\cdot\nabla v_{k,n}\,dx dt\to \int_0^T\int_{\R^d}\,\nabla v\cdot\nabla v_k\,dx dt\quad (n\to\infty)\,.
\eq
Arguing similarly, we also obtain
\bq
\label{k,n4}
\int_0^T\int_{\R^d}\,\partial_t v\partial_t v_{k,n}\,dx dt\to \int_0^T\int_{\R^d}\,\partial_t v\partial_t v_k\,dx dt\quad (n\to\infty)\,.
\eq
Now we will show that
\bq
\label{k,n5}
\int_0^T\int_{\R^d}\,f(v) v_{k,n}\,dx dt\to \int_0^T\int_{\R^d}\,f(v) v_{k}\,dx dt\quad (n\to\infty)\,.
\eq
To do so, let us remark first that the sequence $(v_{k,n})_{n}$ converges strongly to $v_k$ in $L^2(\R^+\times\R^d)$ and
almost everywhere. Moreover
$$
|f(v) v_{k,n}|\leq 2k|f(v)|\in L^1(\R^+\times\R^d)\,.
$$
This implies the desired convergence \eqref{k,n5} thanks to Lebesgue theorem.

For the right-hand side term in \eqref{k,n} similar arguments give
\bq
\label{k,n6}
\int_{\R^d}\left(\partial_t v\, v_{k,n}(T)-\partial_t v\, v_{k,n}(0)\right)\,dx dt\to\int_{\R^d}\left(\partial_t v\, v_{k}(T)-
\partial_t v\, v_{k}(0)\right)\,dx dt\quad(n\to\infty)\,.
\eq
Letting $n\to\infty$ in \eqref{k,n} we infer
\bq
\label{k,k}
\int_0^T\int_{\R^d}\,\left(\nabla v\cdot\nabla v_{k}-\partial_t v\,\partial_t v_{k}+f(v) \,v_{k}\right)\,dx dt=
\int_{\R^d}\left(\partial_t v\, v_{k}(T)-\partial_t v\, v_{k}(0)\right)\,dx dt\,.
\eq
Our aim now is to pass to the limit in \eqref{k,k} as $k$ goes to infinity in order to obtain the desired identity
 \eqref{weak1}. To this end observe that $\nabla v_k=\beta_k'(v)\nabla v$, $\partial_t v_k=\beta_k'(v)\partial_t v$
  and $\beta_k'(s)\to 1$ as $k$ tends to infinity. It follows that
\bq
\label{k,k1}
\int_0^T\int_{\R^d}\,\left(\nabla v\cdot\nabla v_{k}-\partial_t v\,\partial_t v_{k}\right)\,dx dt\to
 \int_0^T\,\int_{\R^d}\,\left(|\nabla v|^2-|\partial_t v|^2\right)\,dx\,dt\quad(k\to\infty)\,.
\eq
In other respects, using the fact that $v_k\to v$ almost everywhere and the monotone convergence theorem, we infer
\bq
\label{k,k2}
\int_0^T\int_{\R^d}\,v_k f(v)\,dx dt\to \int_0^T\,\int_{\R^d}\,v f(v)\,dx\,dt\quad(k\to\infty)\,.
\eq
The last term in \eqref{k,k} can be handled in a similar way and we obtain finally \eqref{weak1}.
\end{proof}

 It will be useful later to remark that assumption \eqref{H21} is equivalent to the following estimate
 \bq
 \label{H11}
 F(u+w)-F(u)-f(u)w\geq -C(R)|w|^2\;\;\mbox{for all}\;\; w\;\;\mbox{whenever}\;\; |u|\leq R\,
 \eq
and that the assumption \eqref{H2} yields
 \bq
 \label{H22}
 |f(u+w)-f(u)-f'(u)w|\leq C(R)\left(|w|^2+|w|^{2^*}\right)\;\;\mbox{for all}\;\; w\;\;\mbox{whenever}\;\; |u|\leq R,
 \eq
where $C(R)>0$ is a constant only depending on $R$. Since obviously \eqref{H1} implies \eqref{H21}, we also have
 estimate \eqref{H11} under assumption \eqref{H1}. Precisely, we have
\begin{lem}
We have the following assertions:
\begin{itemize}
\item \eqref{H1}$\Longrightarrow$\eqref{H21}$\Longleftrightarrow$\eqref{H11}.
\item \eqref{H2}$\Longrightarrow$\eqref{H22}.
\end{itemize}
\end{lem}
\begin{proof}
First let us show that \eqref{H1}$\Longrightarrow$\eqref{H21}$\Longrightarrow$\eqref{H11}.
Since the assumption \eqref{H1} implies that $F$ is always positive, the first implication follows.
To prove \eqref{H11} under \eqref{H21},
remark that, for $|w|\geq 1$ and $|u|\leq R$,
$$
f(u)w+F(u)-F(u+w)\leq C(R)|w|^2\,.
$$
For $|w|\leq 1$, we use Taylor's expansion
$$
F(u+w)-F(u)-F'(u)w=\frac{w^2}{2}F^{''}(u+\theta w),\quad 0\leq\theta\leq 1\,.
$$
Hence $F(u+w)-F(u)-F'(u)w\geq C_R |w|^2$ with $C(R)=\frac{1}{2}\,\displaystyle\inf_{|s|\leq R+1}\,F^{''}(s)$.\\

Now assume \eqref{H2}.  For $|w|\leq 1$ we obtain estimate \eqref{H22} by Taylor's expansion.
For $|w|\geq 1$ and $|u|\leq R$, we have just to remark that, for some $\eta>0$,
$$
|f(u+w)|\leq C|u+w|^{2^*-\eta}\leq C(R)|w|^{2^*}\,.
$$
\end{proof}

\begin{proof}[Proof of Theorem \ref{main-wave}] We will only consider forward time $t\geq 0$.
Set $w:=v-u$ and observe that
$$
\Box w+ f(u+w)-f(u)=0\,.
$$
Expand
\bq
\label{expan}
E(v)=E(u)+ I(t)+J(t),
\eq
where
\beqn
I(t)&=&\int_0^t\int_{\R^d}\,\Big(f(u)+f'(u) w-f(u+w)\Big) \partial_t u\,dx ds,\\\\
J(t)&=&\int_{\R^d}\,\left(\frac{1}{2}|{\mathbf D} w|^2+F(u+w)-F(u)-f(u) w\right)\,dx\,.
\eeqn
We refer to \cite{Struwe-IMRN} for a rigorous derivation of formula \eqref{expan}.
Since $E(v(t))\leq E(v(0))$ and $E(u(t))=E(u(0))$, we deduce that
\bq
\label{main1}
E(u(t))+I(t)+J(t)\leq E(u(0))\,.
\eq
Hence, by \eqref{H11}, we obtain
\bq
\label{main2}
\frac{1}{2}\int_{\R^d}\,|{\mathbf D} w(t)|^2\,dx\leq C\|w(t)\|_{L^2}^2+\int_0^t\int_{\R^d}\,\Big(f(u)+
f'(u) w-f(u+w)\Big) \partial_t u\,dx ds\,.
\eq

To conclude the proof of the theorem, we shall distinguish between two cases.
First we investigate the (simplest) case when the nonlinearity $f$
 satisfies the assumption \eqref{H2}-\eqref{H21}. As pointed out above, we have the estimate \eqref{H22}. Hence
\bq
\label{H23}
\frac{1}{2}\|{\mathbf D} w(t)\|_{L^2}^2\leq C\|w(t)\|_{L^2}^2+C\int_0^t\,\Big(\|w(s)\|_{L^2}^2+\|w(s)\|_{L^{2^*}}^2\Big)\,ds\,.
\eq
Using the fact that
\bq
\label{H24}
\|w(s)\|_{L^{2^*}}\leq C \|\nabla w(s)\|_{L^2},
\eq
and
\bq
\label{H25}
\|w(t)\|_{L^2}^2+\int_0^t\,\|w(s)\|_{L^2}^2\,ds\leq C\int_0^t\,\|{\mathbf D} w(\tau)\|_{L^2}^2\,d\tau,
\eq
we deduce that
$$
\|{\mathbf D} w(t)\|_{L^2}^2\leq C\int_0^t\,\|{\mathbf D} w(s)\|_{L^2}^2\,ds,
$$
and the conclusion follows by Gronwall's inequality. Now we consider the case when the nonlinearity $f$ satisfies
 the assumption \eqref{H1}. We shall make
 use of the following Lemma.
\begin{lem}
\label{main33}
We have
\beq
\label{main3}
\nonumber
\int_0^t\int_{\R^d}\,\Big(f(u)+f'(u) w-f(u+w)\Big) \partial_t u\,dx ds&\leq& C\int_0^t\,\|w(s)\|_{L^2}^2\,ds\\&+
&  C\int_0^t\,\int_{\R^d}\, w\left(f(u+w)-f(u)\right)\,dx\,ds.
 \eeq
 \end{lem}

Let us admit this lemma for a moment and continue the proof of Theorem \ref{main-wave}.
Plugging \eqref{main3} into \eqref{main2} we find
\beq
\label{main4}
\nonumber
\frac{1}{2}\int_{\R^d}\,|{\mathbf D} w(t)|^2\,dx&\leq& C\|w(t)\|_{L^2}^2+C\int_0^t\,\|w(s)\|_{L^2}^2\,ds\\&+
&C\int_0^t\,\int_{\R^d}\, w\left(f(u+w)-f(u)\right)\,dx\,ds.
\eeq
Applying Proposition \ref{weak} to $v$ we obtain
\beq
\label{main5}
\nonumber
\int_0^t\,\int_{\R^d}\, w\left(f(u+w)-f(u)\right)\,dx\,ds&=
&\int_0^t\,\int_{\R^d}\,\left(|\partial_t w|^2-|\nabla w|^2\right)\,dx\,ds+
\int_{\R^d}\,\partial_t w(t)\,w(t)\,dx\\&\leq&\int_0^t\,\|{\mathbf D} w(s)\|_{L^2}^2\,ds+
\frac{1}{4C}\,\|{\mathbf D} w(t)\|_{L^2}^2+C\| w(t)\|_{L^2}^2\,.
\eeq
Therefore, by \eqref{main4} we have

\beq
\label{main6}
\nonumber
\frac{1}{4}\int_{\R^d}\,|{\mathbf D} w(t)|^2\,dx&\leq&C\|w(t)\|_{L^2}^2+C\int_0^t\,\|w(s)\|_{L^2}^2\,ds\\ &+
&C\int_0^t\,\|{\mathbf D} w(s)\|_{L^2}^2\,ds\,.
\eeq
The desired conclusion follows from Gronwall's inequality.\end{proof}

It remains to prove Lemma \ref{main33}.

\begin{proof}[Proof of Lemma \ref{main33}]
Recall that
\bq
\label{main7}
I(t)=\int_0^t\int_{\R^d}\,\Big(f(u)+f'(u) w-f(u+w)\Big) \partial_t u\,dx ds,
\eq
and write
\beq
\label{main8}
\nonumber
I(t)&\leq&\int_0^t\int_{\R^d}\,\Big|\Big(f(u)+f'(u) w-f(u+w)\Big) \partial_t u\Big|{\mathbf 1}_{\{|w|>1\}}\,dx ds\\ &+
&\int_0^t\int_{\R^d}\,\Big|\Big(f(u)+f'(u) w-f(u+w)\Big) \partial_t u\Big|{\mathbf 1}_{\{|w|\leq 1\}}\,dx ds\\
\nonumber&\leq& I_1(t)+I_2(t)\,.
\eeq
Since $u$ is supported in a fixed compact subset of $[0,T]\times\R^d$, then
\bq
\label{main88}
I_2(t)\leq C\int_0^t\,\|w(s)\|_{L^2}^2\,ds,
\eq
and
\bq
\label{main9}
I_1(t)\leq C\int_0^t\,\|w(s)\|_{L^2}^2\,ds+C\int_0^t\int_{\R^d}\,\Big|f(u+w) w\Big|{\mathbf 1}_{\{|w|>1\}}\,dx ds,
\eq
with $C=C(u)$. Define
$$
A^{+}=\Big\{ (t,x);\quad w f(u+w)\geq 0\Big\},\quad A^{-}=\Big\{ (t,x);\quad w f(u+w)< 0\Big\},
$$
and write
\beqn
\int_0^t\int_{\R^d}\,\Big|f(u+w) w\Big|{\mathbf 1}_{\{|w|>1\}}\,dx ds&=
&\int_0^t\int_{\R^d}\,w f(u+w) {\mathbf 1}_{\{|w|>1\}}{\mathbf 1}_{A^{+}}\,dx ds\\ &+
&\int_0^t\int_{\R^d}\,- w f(u+w) {\mathbf 1}_{\{|w|>1\}}{\mathbf 1}_{A^{-}}\,dx ds\,.
\eeqn
Observe that $w$ is bounded on the set $A^{-}$. Therefore, for $|w|\geq 1$, we obtain
\beqn
-w f(u+w)&=&w f(u+w)-2w f(u+w)\\
&=&w\left(f(u+w)-f(u)\right)+w f(u)-2w f(u+w)\\
&\leq& w f(u+w)+C |w|^2\,.
\eeqn
This gives
\bq
\label{K1}
I_1(t)\leq C\int_0^t\,\|w(s)\|_{L^2}^2\,ds+\int_0^t\,\int_{\R^d}\,w\left(f(u+w)-f(u)\right)\,dx ds\,.
\eq
Plugging Estimates \eqref{main8}, \eqref{main88} and \eqref{K1} together we find \eqref{main3} as desired.\end{proof}


\section{Nonlinear Schr\"odinger Equations}

In this section we consider the initial value problem for the semilinear Schr\"odinger equation
\bq
\label{nls}
i\partial_t u -\Delta u+f(u)=0 \quad\mbox{on}\quad \R\times\R^d,\;\;d\geq 2
\eq
with Cauchy data
\bq
\label{nlsCD}
u(0,\cdot)=u_0\in {\mathcal D}(\R^d)\,.
\eq
We assume that the nonlinearity $f$ takes the form $f(u)=u F'(|u|^2/2)$ for some smooth function
$F:\R_+\to\R$ satisfying either
\bq
\label{coercive}
0\leq \sqrt{s}\,F'(s)\leq C F(s),\;\;\;\forall\;\;\;s\geq 0,
\eq
or
\bq
\label{NLS1}
|f(u)|\,\leq\,C|u|^q\quad\mbox{for some}\quad q<2^*:=\frac{2d}{d-2},
\eq
and
\bq
\label{NLS2}
F\left(\frac{|u|^2}{2}\right)\geq -C|u|^2\,.
\eq
Note that the growth \eqref{coercive} is not allowed by the assumption $(3.3)$ in \cite{Struwe-IMRN}.
For example, the function $F(s):={\rm e}^{\sqrt{1+2s}}-{\rm e}$ satisfies \eqref{coercive} but not $(3.3)$ in \cite{Struwe-IMRN} which only allows polynomial growth.\\

 Solutions of \eqref{nls} formally satisfy the conservation of mass and Hamiltonian
\bq
\label{mass}
\|u(t)\|_{L^2}=\|u(0)\|_{L^2},
\eq

\bq
\label{Ham} H(u(t)):=\left\|\nabla u(t)\right\|_{L^2}^2+\int_{\R^d}\, F(|u(t)|^2/2)\,dx=H(u(0))\;.
\eq
In contrast to the case of wave equation where the finite speed of propagation is available, classical solutions to
\eqref{nls} with data having compact support need no longer have spatially compact support. Nevertheless, it is possible to construct solutions that are bounded
 in any Sobolev space $H^s$ (see Strauss \cite{Strauss}). \\

As for nonlinear wave equation, for the supercritical growth, it is not known if the Cauchy problem
\eqref{nls}-\eqref{nlsCD} admit global smooth solutions.
However, assuming \eqref{coercive} or \eqref{NLS1}-\eqref{NLS2}, we can construct a global weak
solution $u$ (as in Definition \ref{solution}) with $\nabla u \in L^\infty(\R, L^2(\R^d))$, $F(|u|^2/2)\in L^\infty(\R, L^1(\R^d))$,
satisfying \eqref{nls} in the sense of distributions and the energy inequality
\bq
\label{energy2}
E(u(t))\leq E(u(0)),\quad \forall\;\;t\in\R\,.
\eq
The proof can be done in a similar way as in Strauss \cite{Strauss} (Theorem 3.1).
As for wave equation, this implies that the map $t\longmapsto \nabla u(t)\in L^2(\R^d)$ is weakly continuous. Note that there are almost parallel stories for the nonlinear Schr\"odinger equations (see \cite{CIMM} and references therein).\\

Our uniqueness result can be stated as follows.
\begin{thm}
\label{mainNLS}
Suppose $u\in \displaystyle\bigcap_{s\geq 0}\,L^\infty([-T,T], H^s(\R^d))$ solves \eqref{nls}-\eqref{nlsCD},
where $f$ satisfies \eqref{coercive} or \eqref{NLS1}-\eqref{NLS2}. Also, let $v$ be a weak solution to \eqref{nls}-\eqref{nlsCD}, satisfying the energy inequality
\eqref{energy2}. Then $u=v$ on $[-T,T]$.
\end{thm}

\begin{proof}
We argue exactly as for wave equation and we only consider forward time $t\geq 0$. Denoting by $w=v-u$, then
\bq
\label{diffe}
{\rm i}\partial_t w-\Delta w+f(u+w)-f(u)=0\quad \mbox{on}\quad \R\times \R^d\,.
\eq
Expanding as in the second section, we find
\bq
\label{expand}
E(v(t))=E(u(t))+I(t)+J(t),
\eq
where
\begin{eqnarray*}
I(t)&=&-\int_0^t\int_{\R^d}\,(f(u+w)-f(u)-Df(u)w)\cdot \partial_t u\,dx\,ds,\\
J(t)&=&\frac{1}{2}\int_{\R^d}\,|\nabla w|^2\,dx+\int_{\R^d}\,\left(F(|u+w|^2/2)-F(|u|^2/2)-f(u)\cdot w\right)\,dx.
\end{eqnarray*}
Here and after we denote by
$$
a\cdot b={\mathcal Re}(a{\bar b})\,.
$$
Since $E(v(t))\leq E(v(0))$, $E(u(t))=E(u(0))$, we deduce that
\begin{eqnarray}
\label{Gronw0}
\nonumber
\frac{1}{2}\int_{\R^d}\,|\nabla w|^2\,dx&\leq& \int_0^t\int_{\R^d}\,(f(u+w)-f(u)-Df(u)w)\cdot \partial_t u\,dx\,ds\\ &+&
\int_{\R^d}\,\left(F(|u|^2/2)+f(u)\cdot w-F(|u+w|^2/2)\right)\,dx.
\end{eqnarray}
According to the assumptions on the nonlinearity $f$, we shall distingush between two cases.
Let us suppose first that $f$ satisfies \eqref{NLS1}-\eqref{NLS2}. Then, using the analogous of \eqref{H11} for
the complex setting, we deduce that
\bq
\label{Gronw1}
\frac{1}{2}\int_{\R^d}\,|\nabla w|^2\,dx\leq C\int_{\R^d}\,|w|^2\,dx+
\int_0^t\int_{\R^d}\,(f(u+w)-f(u)-Df(u)w)\cdot \partial_t u\,dx\,ds\,.
\eq
Note that we have the analogous to \eqref{H22} for the complex setting, namely
\bq
\label{H222}
|f(u+w)-f(u)-Df(u)w|\leq C(R)\left(|w|^2+|w|^{2^*}\right)\;\;\mbox{for all}\;\; w\;\;\mbox{whenever}\;\; |u|\leq R\,.
\eq
Hence
\bq
\label{Gronw2}
\frac{1}{2}\| \nabla w(t)\|_{L^2}^2\leq C \|w(t)\|_{L^2}^2+
C\int_0^t\,\left( \|\nabla w(s)\|_{L^2}^2+ \|w(s)\|_{L^2}^2\right)\,ds\,.
\eq
To conclude the proof in this case it remains to treat carefully the term $\|w(t)\|_{L^2}^2$.
Recall that
\bq
\label{Gronw3}
\frac{d}{dt}\| w\|_{L^2}^2=2\int_{\R^d}\, (f(u)-f(u+w))\cdot ({\rm i}w)\,dx\,.
\eq
To bound the right hand term in \eqref{Gronw3}, we will use an extra cancellation. Remark that for $|w|\leq 1$ we
have $|(f(u)-f(u+w))\cdot ({\rm i}w)|\leq C|w|^2$. For $|w|\geq 1$, write
\bq
\label{Gronw4}
(f(u)-f(u+w))\cdot ({\rm i}w)=f(u)\cdot({\rm i}w)+f(u+w)\cdot ({\rm i}u),
\eq
where we have used the fact that $f(u+w)\overline{u+w}\in\R$. It follows that, for $|w|\geq 1$,
\bq
\label{Gronw5}
|(f(u)-f(u+w))\cdot ({\rm i}w)|\leq C |w|^{2^*},
\eq
and finally
\bq
\label{Gronw6}
|(f(u)-f(u+w))\cdot ({\rm i}w)|\leq C(R)\left(|w|^2+|w|^{2^*}\right)\;\;
\mbox{for all}\;\; w\;\;\mbox{whenever}\;\; |u|\leq R\,.
\eq
Combining \eqref{Gronw3} and \eqref{Gronw6}, we get
\bq
\label{Gronw7}
\|w(t)\|_{L^2}^2\leq C \int_0^t\,\left( \|\nabla w(s)\|_{L^2}^2+ \|w(s)\|_{L^2}^2\right)\,ds\,.
\eq
Adding \eqref{Gronw2}+2C\eqref{Gronw7}, we conclude by applying Gronwall's inequality.\\

Assume now that $f$ satisfies \eqref{coercive}. Hence,
\bq
\label{coercive1}
|f(u)|\leq C F(|u|^2/2).
\eq

\noindent{\bf Claim.} We have
\begin{eqnarray}
\nonumber
\|w(t)\|_{L^2}^2&\leq& C(A+1) \int_0^t\, \|w(s)\|_{L^2}^2\,ds\\ \label{L2} &+&C\int_0^t\int_{\R^d}\Big((A+1)|w|^2+F(|u+w|^2/2)-F(|u|^2/2)-f(u)\cdot w\Big)\,dx\,ds\;.
\end{eqnarray}
for some constants $C=C(u)>0$ and $A>0$ such that
\bq
\label{A}
F(|u+w|^2/2)-F(|u|^2/2)-f(u)\cdot w +(A+1)|w|^2\geq 0,\;\;\;\forall \;\;\;u,w\,.
\eq
Indeed, recalling \eqref{Gronw3} and using the the extra cancellation \eqref{Gronw4}, we get
\beqn
\frac{d}{dt}\| w\|_{L^2}^2&=&2\int_{\R^d}\, (f(u)-f(u+w))\cdot ({\rm i}w)\,dx\\
&=&2\int_{\{|w|<1\}}\, (f(u)-f(u+w))\cdot ({\rm i}w)\,dx+2\int_{\{|w|\geq 1\}}\,\Big( f(u)\cdot ({\rm i}w)+f(u+w)\cdot ({\rm i}u)\Big)\,dx\\
&\leq&C\|w(t)\|_{L^2}^2+C\int_{\{|w|\geq 1\}}\,F(|u+w|^2/2)\,dx.
\eeqn
Upon writting
\beqn
\int_{\{|w|\geq 1\}}\,F(|u+w|^2/2)\,dx&=&\int_{\{|w|\geq 1\}}\,\Big(F(|u+w|^2/2)-F(|u|^2/2)-f(u)\cdot w\Big)\,dx\\&+&
\int_{\{|w|\geq 1\}}\,\Big(F(|u|^2/2)+f(u)\Big)\,dx
\eeqn and observing that the last intgral is less than $A\| w\|_{L^2}^2$ for some $A=A(u)>0$ which can
be taken large enough so that \eqref{A} holds, we obtain \eqref{L2}. Arguing exactly as above, we infer
\begin{eqnarray}
\nonumber
\int_{\R^d}\,\Big(f(u+w)&-&f(u)-Df(u)w\Big)\cdot \partial_t u\,dx\leq\tilde{C}(A+1) \int_0^t\, \|w(s)\|_{L^2}^2\,ds\\ \label{Last} &+&\tilde{C}\int_0^t\int_{\R^d}\Big((A+1)|w|^2+F(|u+w|^2/2)-F(|u|^2/2)-f(u)\cdot w\Big)\,dx\,ds\;.
\end{eqnarray}
Adding \eqref{Gronw0}+2(A+1)\eqref{L2}, we obtain
\begin{eqnarray*}
\hspace{-0.5cm}\frac{1}{2}\|\nabla w(t)\|_{L^2}^2+(A+1)\|w(t)\|_{L^2}^2+\int\left((A+1)|w|^2+F(|u+w|^2/2)-F(|u|^2/2)-f(u)\cdot w\right)dx&\leq&\\C\int_0^t\,(A+1)\|w(s)\|_{L^2}^2\,ds+C\int_0^t\int_{\R^d}\Big((A+1)|w|^2+F(|u+w|^2/2)-F(|u|^2/2)-f(u)\cdot w\Big)\,dx\,ds\,.
\end{eqnarray*}
We conclude by applying Gronwall's inequality.
\end{proof}

\section{Appendix}

We give here a proof of global existence of weak solutions to \eqref{wave1}-\eqref{CD} under the
assumptions \eqref{H2}-\eqref{H21}. As in \cite{Stru}, by assumption \eqref{H21} there exist sequences
$r_k^{\pm}\to\pm\infty$ such that
$$
r_k^{\pm}f(r_k^{\pm})\geq -C|r_k^{\pm}|^2.
$$
We approximate $f$ by Lipschitz functions
\begin{eqnarray*}
 f_k(u)&=&\; \left\{
\begin{array}{cllll}f(r_k^-) \quad&\mbox{if}&\quad
u<r_k^-,\\\\ f(u) \quad
&\mbox{if}&\quad r_k^-\leq u\leq r_k^+,\\\\
f(r_k^+)\quad&\mbox{if}&\quad u>r_k^+.
\end{array}
\right.
\end{eqnarray*}
Clearly the primitive $F_k(u)=\dint_0^u\,f_k(v)\,dv$ of $f_k$ satisfies
\bq
\label{lower}
F_k(u)\geq -C|u|^2.
\eq
Next, we consider the following approximate Cauchy problems
\bq
\label{approx1}
\partial_t^2 v^k-\Delta\,v^k+f_k(v^k)=0,\quad \left(v^k(0,\cdot),\partial_t v^k(0,\cdot)\right)\in {\mathcal D}(\R^d)^2\,.
\eq
Classical existence theory tell us that \eqref{approx1} has a global classical solution $v^k$.
Moreover, by the assumption \eqref{H21} the sequence $(v^k)$ is bounded in the energy norm.
Hence $(v^k)$ converges strongly to some $v$ in $L^2(Q)$ for any compact space-time region $Q\subset \R\times\R^d$.
 This in particular implies that $v^k\to v$ and $f_k(v^k)\to f(v)$ almost everywhere.
  To conclude that $v$ solves \eqref{wave1}-\eqref{CD} in the sense of distribution it remains to prove that
$$
f_k(v^k)\to f(v)\quad\mbox{in}\quad L^1_{loc}(\R\times\R^d)\,.
$$
This will be done via Vitali's theorem. It suffices to show that, for any $\varepsilon>0$, there exists $\delta>0$ such that
$$
|E|<\delta\Longrightarrow\int_{E}\,|f_k(v^k)|\,dx\,dt<\varepsilon,
$$
where $|E|$ denotes the Lebesgue measure of the set $E\subset\R\times\R^d$.
Using assumption \eqref{H2} and H\"older inequality, we get
\begin{eqnarray*}
\int_{E}\,|f_k(v^k)|\,dx\,dt&\leq &C\int_{E}\,|v^k|^{2^*-\eta}\,dx\,dt\\
&\leq &C\left(\int_{E}\,|v^k|^{2^*}\,dx\,dt\right)^{1-\eta/{2^*}}\,|E|^{\eta/{2^*}}\\
&\leq&C|E|^{\eta/{2^*}},
\end{eqnarray*}
and the conclusion follows. Finally, the energy inequality follows from \eqref{lower} and Fatou's lemma.

\end{document}